\documentclass{amsart}
\usepackage{eurosym}
\usepackage{graphicx,amscd,color,amsmath,amsfonts,amssymb,geometry}

\newtheorem{theorem}{Theorem}[section]
\newtheorem{proposition}[theorem]{Proposition}

\newtheorem{remark}[theorem]{Remark}

\newtheorem{lemma}[theorem]{Lemma}

\numberwithin{equation}{section}

\begin{document}
\title[Polynomial and Multilinear Hardy--Littlewood Inequalities]{Polynomial
and Multilinear Hardy--Littlewood Inequalities: Analytical and Numerical
Approaches}
\author[Campos]{J. Campos}
\address{Departamento de Ci\^{e}ncias Exatas\\
\indent Universidade Federal da Para\'{\i}ba \\
\indent 58.297-000 - Rio Tinto, Brazil.}
\email{jamilson@dce.ufpb.b}
\author[Cavalcante]{W. Cavalcante}
\address{Departamento de Matem\'{a}tica\\
\indent  Universidade Federal de Pernambuco\\
\indent50.740-560 - Recife, Brazil.}
\email{wasthenny.wc@gmail.com}
\author[F\'{a}varo]{V. F\'{a}varo}
\address{Faculdade de Matem\'{a}tica\\
\indent  Universidade Federal de Uberl\^{a}ndia\\
\indent 38.400-902 - Uberl\^{a}ndia, Brazil.}
\email{vvfavaro@gmail.com}
\author[N\'{u}\~{n}ez]{D. N\'{u}\~{n}ez-Alarc\'{o}n}
\address{Departamento de Matem\'{a}tica\\
\indent  Universidade Federal de Pernambuco\\
\indent 50.740-560 - Recife, Brazil.}
\email{danielnunezal@gmail.com}
\author[Pellegrino]{D. Pellegrino}
\address{Departamento de Matem\'{a}tica \\
\indent Universidade Federal da Para\'{\i}ba \\
\indent 58.051-900 - Jo\~{a}o Pessoa, Brazil.}
\email{pellegrino@pq.cnpq.br and dmpellegrino@gmail.com}
\author[Serrano]{D. M. Serrano-Rodr\'iguez}
\address{Departamento de Matem\'{a}tica\\
\indent  Universidade Federal de Pernambuco\\
\indent 50.740-560 - Recife, Brazil.}
\email{dmserrano0@gmail.com}
\thanks{J. Campos was supported by a CAPES Postdoctoral scholarship. V. V. F%
\'{ }avaro was supported by FAPEMIG Grants CEX-APQ-01409-12, PPM-00086-14
and CNPq Grants 482515/2013-9, 307517/2014-4. W. Cavalcante is supported by
Capes. D. N\'{u}\~{n}ez and D.M. Serrano were supported by CNPq Grant
461797/2014-3. D. Pellegrino was supported by CNPq}
\maketitle

\begin{abstract}
We investigate the growth of the polynomial and multilinear
Hardy--Littlewood inequalities. Analytical and numerical approaches are
performed and, in particular, among other results, we show that a simple
application of the best known constants of the Clarkson inequality improves
a recent result of Araujo et al. We also obtain the optimal constants of the
generalized Hardy--Littlewood inequality in some special cases.
\end{abstract}

\section{Introduction}

\bigskip The investigation of polynomials and multilinear operators acting
on Banach spaces is a fruitful topic of investigation that dates back to the
$30^{\prime }s$ (see, for instance \cite{bh, hardy, LLL} and, for recent
papers, \cite{rueda, bayy, botelho, dimant, garcia} among many others).

Let $\mathbb{K}$ be the real or complex scalar field, and$\,n\geq 1$ be a
positive integer. In 1930 Littlewood proved his well-known $4/3$ inequality
to solve a problem posed by P.J. Daniell (see \cite{LLL}). The Littlewood's $%
4/3$ inequality asserts that
\begin{equation*}
\left( \sum\limits_{i,j=1}^{n}\left\vert T(e_{i},e_{j})\right\vert ^{\frac{4%
}{3}}\right) ^{\frac{3}{4}}\leq \sqrt{2}\left\Vert T\right\Vert
\end{equation*}%
for all positive integers $n$ and every continuous bilinear form $%
T:c_{0}\times c_{0}\rightarrow \mathbb{K}$, where\linebreak $\Vert T\Vert
:=\sup_{z^{(1)},z^{(2)}\in B_{c_{0}}}|T(z^{(1)},z^{(2)})|$. The exponent $%
4/3 $ is optimal and in the case $\mathbb{K}=\mathbb{R}$ the optimality of
the constant $\sqrt{2}$ is also known (see \cite{diniz}). Soon afterwards
this inequality was generalized by Hardy and Littlewood (\cite{hardy}, 1934)
for bilinear forms on $\ell _{p}$ and, in 1982 Praciano-Pereira (\cite{pra})
extended the result of Hardy and Littlewood to $m$-linear forms on $\ell
_{p} $. Another generalization of the Hardy--Littlewood inequalities for $m$%
-linear forms was obtained by Dimant and Sevilla-Peris, and will be treated
in Remark \ref{988}.

\bigskip The Hardy--Littlewood inequalities for $m$-linear forms is the
following result:

\bigskip

\textbf{Theorem (Hardy--Littlewood/Praciano-Pereira).} Let $m\geq 2$ be a
positive integer. For $p\geq 2m,$ there is a constant $C_{\mathbb{K}%
,m,p}\geq 1$ such that%
\begin{equation*}
\left( \sum_{i_{1},...,i_{m}=1}^{n}\left\vert T(e_{i_{1}},\ldots
,e_{i_{m}})\right\vert ^{\frac{2mp}{mp+p-2m}}\right) ^{\frac{mp+p-2m}{2mp}%
}\leq C_{\mathbb{K},m,p}\left\Vert T\right\Vert ,
\end{equation*}

for all positive integers $n$ and all $m$-linear forms $T:\ell
_{p}^{n}\times \cdots \times \ell _{p}^{n}\rightarrow \mathbb{K}$.

The exponent $\frac{2mp}{mp+p-2m}$ is optimal and $\Vert T\Vert
:=\sup_{z^{(1)},...,z^{(m)}\in B_{\ell _{p}^{n}}}|T(z^{(1)},...,z^{(m)})|\,$%
. In the limiting case ($p=\infty ,$ considering, of course $f(\infty
):=\lim_{p\rightarrow \infty }f(p)$ regardless of the function $f$), we
recover the classical multilinear Bohnenblust--Hille inequality (see \cite%
{bh}). The original upper estimate for $C_{\mathbb{K},m,p}$ is $2^{\frac{m-1%
}{2}}$. Recently, in some papers (see \cite{adv22, ap, ooo}), this estimate
was improved for all $m$ and $p$ with the only exception of the case $C_{%
\mathbb{R},m,2m}$.

The precise behavior of the growth of the optimal constants $C_{\mathbb{K}%
,m,p}$ is still unknown (some partial results can be found in \cite{laa,
adv22, ap}).

Up to now, the best known lower estimates for $C_{\mathbb{R},m,p}$ are
always smaller than $2$ and again the more critical situation is when $p=2m$%
, where the lower estimates presented in \cite{laa} are more difficult to
obtain and not explicitly stated for the case $p=2m$.

In view of the special role played by the constants $C_{\mathbb{R},m,2m}$
and since this case is a kind of dual version of the classical
Bohnenblust--Hille inequality (see details in Section \ref{890}), in the
Sections 3 and 4 we investigate this critical case and obtain quite better
lower estimates. Our approach has two novelties: a new class of multilinear
forms, not investigated before in similar context, and a new numerical
approach in this framework. As it will be clear along the paper the new
family of multilinear forms introduced in this paper is more effective to
obtain good lower estimates for the Hardy--Littlewood inequality.

In Section \ref{ggee} we investigate the generalized Hardy--Littlewood
inequality. Our approach provides new lower bounds for this inequality. As a
consequence of our results, in Theorem \ref{optimal} we obtain optimal
constants for some cases of three-linear forms.

In Section \ref{ii1} we investigate the polynomial Hardy--Littlewood
inequality. The approaches of Sections \ref{ggee} and \ref{ii1} are entirely
analytic and do not depend on computation assistance.

\section{ The multilinear Hardy--Littlewood inequality\label{890}}

\bigskip

From now on, if $p\in \left( 1,2\right) $, $p^{\ast }$ is the extended real
number such that $\frac{1}{p}+\frac{1}{p^{\ast }}=1.$ Also, $E^{\prime }$
denotes the topological dual of a Banach space $E$. By $\mathcal{L}\left(
^{m}E;F\right) $ we denote the Banach space of all (bounded) $m$-linear
operators $U:E\times \cdots \times E\rightarrow F$, with $E$, $F~$Banach
spaces over $\mathbb{K}$. For $1\leq s\leq r<\infty ,\,U\in \mathcal{L}%
\left( ^{m}E;F\right) $ is called \emph{multiple $(r,s)$-summing}, if there
exists a constant $C>0$ such that
\begin{equation*}
\left( \sum_{i_{1},\dots ,i_{m}=1}^{n}\left\Vert U\left( x_{i_{1}},\dots
,x_{i_{m}}\right) \right\Vert _{F}^{r}\right) ^{\frac{1}{r}}\leq C\left\Vert
U\right\Vert \prod_{k=1}^{m}\left\Vert \left( x_{i_{k}}\right)
_{i_{k}=1}^{n}\right\Vert _{w,s}
\end{equation*}%
for all finite choice of vectors $x_{i_{k}}\in E,\,1\leq i_{k}\leq n,\,1\leq
k\leq m$, where%
\begin{equation*}
\Vert \left( x_{i}\right) _{i=1}^{n}\Vert _{w,s}:=\sup_{\Vert \varphi \Vert
_{E^{\prime }}\leq 1}\left( \sum_{i=1}^{n}|\varphi (x_{i})|^{s}\right) ^{%
\frac{1}{s}}.
\end{equation*}%
The vector space of all multiple $(r,s)$-summing operators in $\mathcal{L}%
\left( ^{m}E;F\right) $ is denoted by $\Pi _{(r,s)}\left( ^{m}E;F\right) $.
For more details of the theory of multiple summing operators theory see \cite%
{matos, santosp, per}.

In the terminology of the multiple summing operators, it is well known (see,
for instance, \cite[Section 5]{dimant}) that the
Hardy--Littlewood/Praciano-Pereira inequality is equivalent to the equality

\begin{equation*}
\Pi _{\left( \frac{2mp}{mp+p-2m};p^{\ast }\right) }(^{m}E;\mathbb{K})=%
\mathcal{L}(^{m}E;\mathbb{K}).
\end{equation*}%
In other words, if $m\geq 2$ and $p\geq 2m$, then there is a constant $C_{%
\mathbb{K},m,p}\geq 1$ such that%
\begin{equation*}
\left( \sum_{i_{1},...,i_{m}=1}^{n}\left\vert T(x_{i_{1}},\ldots
,x_{i_{m}})\right\vert ^{\frac{2mp}{mp+p-2m}}\right) ^{\frac{mp+p-2m}{2mp}%
}\leq C_{\mathbb{K},m,p}\left\Vert T\right\Vert
\prod\limits_{k=1}^{m}\left\Vert \left( x_{i_{k}}\right)
_{i_{k}=1}^{n}\right\Vert _{w,p^{\ast }}
\end{equation*}%
for all $m$-linear forms $T:E\times \cdots \times E\rightarrow \mathbb{K}$,
for all finite choice of vectors $x_{i_{k}}\in E,\,1\leq i_{k}\leq n,\,1\leq
k\leq m$.

As mentioned in the introduction, the case $p=2m$ in the Hardy--Littlewood
inequality is specially interesting. In this case we have very few
information on the constants involved, and moreover, this case is a kind of
dual version of the Bohnenblust--Hille inequality, in the sense that in the
pair of parameters $\left( \frac{2mp}{mp+p-2m};p^{\ast }\right) $, each case
has a coordinate which is kept constant (in reverse location). More
specifically, in the terminology of the multiple summing operators, the
Bohnenblust--Hille inequality asserts that
\begin{equation*}
\Pi _{\left( \frac{2m}{m+1};1\right) }(^{m}E;\mathbb{K})=\mathcal{L}(^{m}E;%
\mathbb{K})
\end{equation*}%
for all Banach spaces $E$. On the other hand, when $p=2m$, the
Hardy--Littlewood inequality is equivalent to
\begin{equation*}
\Pi _{\left( 2;\frac{2m}{2m-1}\right) }(^{m}E;\mathbb{K})=\mathcal{L}(^{m}E;%
\mathbb{K})
\end{equation*}%
for all Banach spaces $E$.

Up to now the best known upper estimates for the constants $\left( C_{%
\mathbb{R},m,p}\right) _{m=1}^{\infty }$ can be found in \cite[page 1887]{ap}
and \cite{ooo}. The updated results on the lower bounds for these constants
are:

\begin{itemize}
\item $C_{\mathbb{R},m,p}\geq 2^{\frac{mp+2m-2m^{2}-p}{mp}}$ for $p>2m$ and $%
C_{\mathbb{R},m,p}>1$ for $p=2m$ (see \cite{laa});
\end{itemize}

From now on $p^{\ast }$ denotes the conjugate number of $p$. In this section
we find an overlooked (and simple) connection between the Clarkson's
inequalities and the Hardy--Littlewood's constants which helps to find
analytical lower estimates (without the use of a computational aid) for
these constants.

\begin{theorem}
\label{popkhtr} Let $m\geq 2$ and $p\geq 2m$. The optimal constants of the
Hardy-Littlewood inequalities satisfies
\begin{equation*}
C_{\mathbb{R},m,p}\geq \frac{2^{\frac{2mp+2m-p-2m^{2}}{mp}}}{\sup_{x\in
\lbrack 0,1]}\frac{((1+x)^{p^{\ast }}+(1-x)^{p^{\ast }})^{\frac{1}{p^{\ast }}%
}}{{(1+x^{p})^{1/p}}}}.
\end{equation*}
\end{theorem}

\begin{proof}
For a given Banach space $E$ we know that $\Psi :\mathcal{L}\left( ^{2}E;%
\mathbb{R}\right) \rightarrow \mathcal{L}\left( E;E^{\ast }\right) $ given
by $\Psi (T)(x)(y)=T\left( x,y\right) $ is an isometric isomorphism. For $%
E=\ell _{p}^{2}$ and using the characterization of the dual of $\ell
_{p}^{2} $, we conclude that for the bilinear form
\begin{equation*}
\begin{array}{ccccl}
T_{2,p} & : & \ell _{p}^{2}\times \ell _{p}^{2} & \rightarrow & \mathbb{R}
\\
&  & ((x_{i}^{(1)}),(x_{i}^{(2)})) & \mapsto &
x_{1}^{(1)}x_{1}^{(2)}+x_{1}^{(1)}x_{2}^{(2)}+x_{2}^{(1)}x_{1}^{(2)}-x_{2}^{(1)}x_{2}^{(2)},%
\end{array}%
\end{equation*}%
we have
\begin{equation*}
\begin{array}{ccccl}
\Psi (T_{2,p}) & : & \ell _{p}^{2} & \rightarrow & \ell _{p^{\ast }}^{2} \\
&  & (x_{i}) & \mapsto & (x_{1}+x_{2},x_{1}-x_{2}).%
\end{array}%
\end{equation*}%
Since $p\geq 2m$ and $m\geq 2$, using the best constants from the Clarkson's
inequality in the real case (see \cite[Theorem 2.1]{mali}) we know the norm
of the linear operator $\Psi (T_{2,p})$ (and consequently the norm of the
bilinear form $T_{2,p}$), i.e.,
\begin{equation*}
\Vert T_{2,p}\Vert =\left\Vert \Psi (T_{2,p})\right\Vert =\sup_{x\in \lbrack
0,1]}\frac{((1+x)^{p^{\ast }}+(1-x)^{p^{\ast }})^{\frac{1}{p^{\ast }}}}{%
(1+x^{p})^{1/p}}.
\end{equation*}%
Now, as in \cite{laa}, we define inductively
\begin{equation*}
\begin{array}{rccl}
T_{m,p}: & \ell _{p}^{2^{m-1}}\times \cdots \times \ell _{p}^{2^{m-1}} &
\rightarrow & \mathbb{R} \\
& (x^{(1)},...,x^{(m)}) & \mapsto &
(x_{1}^{(m)}+x_{2}^{(m)})T_{m-1,p}(x^{(1)},...,x^{(m)}) \\
&  &  &
+(x_{1}^{(m)}-x_{2}^{(m)})T_{m-1,p}(B^{2^{m-1}}(x^{(1)}),...,B^{2}(x^{(m-1)})),%
\end{array}%
\end{equation*}%
where $x^{(k)}=(x_{j}^{(k)})_{j=1}^{{2^{m-1}}}\in \ell _{p}^{2^{m-1}}$, $%
1\leq k\leq m$, and $B$ is the backward shift operator in $\ell
_{p}^{2^{m-1}}$ and, again as in \cite{laa}, we conclude that
\begin{align*}
|T_{m,p}(x^{(1)},...,x^{(m)})|& \leq
|x_{1}^{(m)}+x_{2}^{(m)}||T_{m-1,p}(x^{(1)},...,x^{(m)})| \\
&
+|x_{1}^{(m)}-x_{2}^{(m)}||T_{m-1,p}(B^{2^{m-1}}(x^{(1)}),B^{2^{m-2}}(x^{(2)}),...,B^{2}(x^{(m-1)}))|
\\
& \leq \Vert T_{m-1,p}\Vert
(|x_{1}^{(m)}+x_{2}^{(m)}|+|x_{1}^{(m)}-x_{2}^{(m)}|) \\
& \leq 2\Vert T_{m-1,p}\Vert \Vert x^{(m)}\Vert _{p},
\end{align*}%
i.e.,
\begin{equation*}
\Vert T_{m,p}\Vert \leq 2^{m-2}\Vert T_{2,p}\Vert .
\end{equation*}
\end{proof}

Now we have
\begin{equation*}
(4^{m-1})^{\frac{mp+p-2m}{2mp}}=\left(
\sum_{j_{1},...,j_{m}=1}^{2^{m-1}}\left\vert
T_{m,p}(e_{j_{1}},...,e_{j_{m}})\right\vert ^{\frac{2mp}{mp+p-2m}}\right) ^{%
\frac{mp+p-2m}{2mp}}\leq C_{\mathbb{R},m,p}2^{m-2}\Vert T_{2,p}\Vert
\end{equation*}%
and thus
\begin{equation*}
C_{\mathbb{R},m,p}\geq \frac{(4^{m-1})^{\frac{mp+p-2m}{2mp}}}{2^{m-2}\Vert
T_{2,p}\Vert }=\frac{2^{2\left( m-1\right) \left( \frac{mp+p-2m}{2mp}\right)
-\left( m-2\right) }}{\sup_{x\in \lbrack 0,1]}\frac{\left( \left( 1+x\right)
^{p^{\ast }}+\left( 1-x\right) ^{p^{\ast }}\right) ^{1/p^{\ast }}}{\left(
1+x^{p}\right) ^{1/p}}}
\end{equation*}

When $m=2,$ using estimates of \cite[page 1369]{mali}, note that%
\begin{align*}
C_{\mathbb{R},2,4}& \geq \frac{2}{\sqrt{3}}>1.1546 \\
C_{\mathbb{R},2,8}& \geq \frac{2^{\frac{5}{4}}}{1.892}>1.2570 \\
C_{\mathbb{R},2,p}& \geq \frac{2^{\frac{2mp+2m-p-2m^{2}}{mp}}}{1.9836}>1.3591%
\text{ for }p=1+\log _{9/10}1/19 \\
C_{\mathbb{R},2,p}& \geq \frac{2^{\frac{2mp+2m-p-2m^{2}}{mp}}}{1.9999}>1.4105%
\text{ for }p=1+\log _{99/100}1/199.
\end{align*}%
Using the old estimates of \cite{laa} for $p>2m$ (i.e., $C_{\mathbb{R}%
,m,p}\geq 2^{\frac{mp+2m-2m^{2}-p}{mp}}$) we can easily verify that the old
estimates are worse. Also, in the old estimates we have no closed formula
for the case $p=2m.$

\begin{remark}
One may try to use the complex Clarkson's inequalities to obtain nontrivial
lower bounds for the constants of the complex Hardy-Littlewood inequality.
But, this is not effective, since we just get trivial lower bounds, i.e., $1$%
.
\end{remark}

\begin{remark}[The case $m<p<2m$]
\label{988}There is also a version of Hardy--Littlewood's inequality for $%
m<p<2m,$ due to Dimant and Sevilla-Peris (\cite{dimant} and the forthcoming
Section 6). In this case, the optimal exponent is $\frac{p}{p-m}$ and we
still denote the optimal constant for this inequality by $C_{\mathbb{K},m,p}$%
. The best information we have so far for the lower estimates for the
constant $C_{\mathbb{R},m,p}$ are trivial, that is,
\begin{equation*}
1\leq C_{\mathbb{R},m,p}\leq (\sqrt{2})^{m-1}.
\end{equation*}%
Similarly to the argument used in the proof of the Theorem \ref{popkhtr}, we
can also provide a closed formula (which depends on $p$) for the lower
bounds of $C_{\mathbb{R},m,p}$, but in this case, we do not always have
nontrivial information. More precisely, we prove that
\begin{equation*}
C_{\mathbb{R},m,p}\geq \frac{2^{\frac{mp+2m-2m^{2}}{p}}}{\sup_{x\in \lbrack
0,1]}\frac{((1+x)^{p^{\ast }}+(1-x)^{p^{\ast }})^{\frac{1}{p^{\ast }}}}{%
(1+x^{p})^{\frac{1}{p}}}}.
\end{equation*}%
It is important to mention this case because, for suitable choices of $p$,
we get nontrivial lower estimates for $C_{\mathbb{R},m,p}$. For instance,
\begin{equation*}
C_{\mathbb{R},2,7/2}\geq 1.104,\ \ C_{\mathbb{R},3,28/5}\geq 1.025,\ \ \text{%
and}\ \ C_{\mathbb{R},100,199999/1000}\geq 1.003.
\end{equation*}%
This leads us to question the following: Would also be the optimal constants
of the Hardy--Littlewood inequality for $m<p<2m$ strictly greater than $1$?
\end{remark}

\bigskip

\section{First numerical estimates (using well-known multilinear forms)}

\bigskip Since the publication of \cite{diniz}, the family of $m$-linear
forms $T_{m}:\ell _{\infty }\times \cdots \times \ell _{\infty }$ defined
inductively by
\begin{equation}
T_{2}(x,y)=x_{1}y_{1}+x_{1}y_{2}+x_{2}y_{1}-x_{2}y_{2},  \label{t2}
\end{equation}%
\begin{align}
T_{3}(x,y,z)& =(z_{1}+z_{2})\left(
x_{1}y_{1}+x_{1}y_{2}+x_{2}y_{1}-x_{2}y_{2}\right)  \label{t3} \\
& +(z_{1}-z_{2})\left( x_{3}y_{3}+x_{3}y_{4}+x_{4}y_{3}-x_{4}y_{4}\right) ,
\notag
\end{align}%
\begin{align}
T_{4}(x,y,z,w)& =\left( w_{1}+w_{2}\right) \left(
\begin{array}{c}
(z_{1}+z_{2})\left( x_{1}y_{1}+x_{1}y_{2}+x_{2}y_{1}-x_{2}y_{2}\right) \\
+(z_{1}-z_{2})\left( x_{3}y_{3}+x_{3}y_{4}+x_{4}y_{3}-x_{4}y_{4}\right)%
\end{array}%
\right)  \label{t4} \\
& +\left( w_{1}-w_{2}\right) \left(
\begin{array}{c}
(z_{3}+z_{4})\left( x_{5}y_{5}+x_{5}y_{6}+x_{6}y_{5}-x_{6}y_{6}\right) \\
+(z_{3}-z_{4})\left( x_{7}y_{7}+x_{7}y_{8}+x_{8}y_{7}-x_{8}y_{8}\right)%
\end{array}%
\right)  \notag
\end{align}%
and so on, have been used to find lower estimates for Bohnenblust--Hille and
related inequalities (se also \cite{ooo}). In the context of the
Hardy--Littlewood inequalities we also have good results, but in the next
section we invent different multilinear forms that, in our context, provide
better estimates.

The numerical issue involved to obtain our estimates is the calculus of $%
\Vert T_{m}\Vert $ when $\ell _{\infty }$ is replaced by $\ell _{p}$ (in
this case we write $T_{m,p}$ instead of $T_{m}$)$.$ This task refers to a
typical nonlinear optimization problem subject to restrictions. Namely, we
want to find a global maximum of $|T_{m,2m}(x^{(1)},\ldots ,x^{(m)})|$ with $%
x^{(i)}\in B_{\ell _{2m}}$, $i=1,\ldots ,m$ for the operators (\ref{t2}), (%
\ref{t3}), (\ref{t4}), etc.

To perform this computer-aided calculus we use a couple of software:
multi-paradigm numerical computing environment called MATLAB (MATrix
LABoratory) (see \cite{gilat}) to specify the problem and a software library
for large-scale nonlinear optimization called Interior Point to solve it.
Mathematical details of the algorithm used by interior-point can be found in
several publications (see for instance \cite{Byrd,Byrd2,Waltz}).

As the interior-point algorithm is designed to find local solutions for a
given optimization problem starting from a initial data, it is necessary to
find all local solutions (all maxima) and take the greatest of them. This
can be done taking a reasonable distribution of starting points throughout
the domain of the operator.

Performing these calculations for $T_{m,2m}$, we obtain%
\begin{equation}
\begin{tabular}{|c|cl|}
\hline
& \multicolumn{1}{|c|}{} &  \\
$C_{\mathbb{R},2,4}>$ & $\frac{2}{1.74}$ & \multicolumn{1}{|l|}{$>1.149$} \\
&  & \multicolumn{1}{|l|}{} \\
$C_{\mathbb{R},3,6}>$ & $\frac{4}{3.29}$ & \multicolumn{1}{|l|}{$>1.215$} \\
& \multicolumn{1}{|c|}{} & \multicolumn{1}{|l|}{} \\
$C_{\mathbb{R},4,8}>$ & \multicolumn{1}{|c|}{$\frac{8}{6.40}$} & $>1.250$ \\
& \multicolumn{1}{|c|}{} &  \\
$C_{\mathbb{R},5,10}>$ & \multicolumn{1}{|c|}{$\frac{16}{12.60}$} & $>1.269$
\\
& \multicolumn{1}{|c|}{} &  \\
$C_{\mathbb{R},6,12}>$ & \multicolumn{1}{|c|}{$\frac{32}{25.00}$} & $>1.280$
\\
& \multicolumn{1}{|c|}{} &  \\
$C_{\mathbb{R},7,14}>$ & \multicolumn{1}{|c|}{$\frac{64}{49.47}$} & $>1.293$
\\
& \multicolumn{1}{|c|}{} &  \\
$C_{\mathbb{R},8,16}>$ & \multicolumn{1}{|c|}{$\frac{128}{98.36}$} & $>1.301$
\\
& \multicolumn{1}{|c|}{} &  \\
$C_{\mathbb{R},9,18}>$ & \multicolumn{1}{|c|}{$\frac{256}{195.81}$} & $%
>1.\,\allowbreak 307.$ \\
& \multicolumn{1}{|c|}{} &  \\ \hline
\end{tabular}
\label{dell99}
\end{equation}

\section{New multilinear forms and better estimates\label{ujp}}

Up to now the best known multilinear forms to use in order to find lower
bounds for the Bohnenblust--Hille and Hardy--Littlewood inequalities were
those defined in (\ref{t2}), (\ref{t3}), (\ref{t4}) and so on. Now we show
that for $m=4,8,16,...$ we get better estimates using slightly different
multilinear forms and numerical computation. Define
\begin{equation*}
\widetilde{T}_{2}(x,y)=x_{1}y_{1}+x_{1}y_{2}+x_{2}y_{1}-x_{2}y_{2},
\end{equation*}%
\begin{align*}
\widetilde{T}_{4}(x,y,z,w)& =\left(
x_{1}y_{1}+x_{1}y_{2}+x_{2}y_{1}-x_{2}y_{2}\right) \left(
z_{1}w_{1}+z_{1}w_{2}+z_{2}w_{1}-z_{2}w_{2}\right) \\
& +\left( x_{1}y_{1}+x_{1}y_{2}+x_{2}y_{1}-x_{2}y_{2}\right) \left(
z_{3}w_{3}+z_{3}w_{4}+z_{4}w_{3}-z_{4}w_{4}\right) \\
& +\left( x_{3}y_{3}+x_{3}y_{4}+x_{4}y_{3}-x_{4}y_{4}\right) \left(
z_{1}w_{1}+z_{1}w_{2}+z_{2}w_{1}-z_{2}w_{2}\right) \\
& -\left( x_{3}y_{3}+x_{3}y_{4}+x_{4}y_{3}-x_{4}y_{4}\right) \left(
z_{3}w_{3}+z_{3}w_{4}+z_{4}w_{3}-z_{4}w_{4}\right) ,
\end{align*}

\begin{eqnarray*}
\widetilde{T}_{8}(x,y,z,w,r,s,t,u) &=&\widetilde{T}_{4}(x,y,z,w)\widetilde{T}%
_{4}(r,s,t,u) \\
&&+\widetilde{T}_{4}(x,y,z,w)\widetilde{T}_{4}(B^{4}\left( r\right)
,B^{4}\left( s\right) ,B^{4}\left( t\right) ,B^{4}\left( u\right) ) \\
&&+\widetilde{T}_{4}(B^{4}(x),B^{4}(y),B^{4}(z),B^{4}(w))\widetilde{T}%
_{4}(r,s,t,u) \\
&&-\widetilde{T}_{4}(B^{4}(x),B^{4}(y),B^{4}(z),B^{4}(w))\widetilde{T}%
_{4}(B^{4}\left( r\right) ,B^{4}\left( s\right) ,B^{4}\left( t\right)
,B^{4}\left( u\right) ),
\end{eqnarray*}%
and so on (recall that $B^{4}$ is the shift operator, as defined before).
Using $\widetilde{T}_{4},\widetilde{T}_{8},$ etc, we obtain%
\begin{equation}
\begin{tabular}{|c|c|l|}
\hline
&  &  \\
$C_{\mathbb{R},4,8}>$ & $\frac{2^{3}}{6.20}$ & $>1.290$ \\
&  &  \\
$C_{\mathbb{R},8,16}>$ & $\frac{2^{7}}{91.48}$ & $>1.399$ \\
&  &  \\
$C_{\mathbb{R},16,32}>$ & $\frac{2^{15}}{22137.70}$ & $>1.\,\allowbreak 480,$
\\
&  &  \\ \hline
\end{tabular}
\label{t44}
\end{equation}

and this procedure seems clearly better than the former.

\section{On the generalized Hardy--Littlewood inequality\label{ggee}}

The main goal in this section is to provide optimal constants for some cases
of three-linear forms in the recently extended version of the
Hardy--Littlewood inequality, presented in \cite{n, dimant}: \bigskip

\textbf{Theorem (Generalized Hardy--Littlewood inequality).} If\textit{\ }$%
m\geq 2$ is a positive integer,\linebreak $2m\leq p\leq \infty $ \textit{and}
$\mathbf{q}:=(q_{1},...,q_{m})\in \left[ \frac{p}{p-m},2\right] ^{m}$ are
\textit{such that}
\begin{equation}
\frac{1}{q_{1}}+\cdots +\frac{1}{q_{m}}=\frac{mp+p-2m}{2p},  \label{ppoo}
\end{equation}%
t\textit{hen there exists a constant} $C_{m,p,\mathbf{q}}^{\mathbb{K}}\geq 1$
\textit{such that}
\begin{equation}
\left( \sum_{j_{1}=1}^{n}\left( \sum_{j_{2}=1}^{n}\left( \cdots \left(
\sum_{j_{m}=1}^{n}\left\vert T(e_{j_{1}},...,e_{j_{m}})\right\vert
^{q_{m}}\right) ^{\frac{q_{m-1}}{q_{m}}}\cdots \right) ^{\frac{q_{2}}{q_{3}}%
}\right) ^{\frac{q_{1}}{q_{2}}}\right) ^{\frac{1}{q_{1}}}\leq C_{m,p,\mathbf{%
q}}^{\mathbb{K}}\left\Vert T\right\Vert  \label{g33}
\end{equation}%
\textit{for all continuous }$m$\textit{--linear forms }$T:\ell
_{p}^{n}\times \cdots \times \ell _{p}^{n}\rightarrow \mathbb{K}$\textit{\
and all positive integers }$n.$\textit{\ } %

\bigskip The case $p=\infty $ recovers the so called generalized
Bohnenblust--Hille inequality (see \cite{alb}) and when $p=\infty $ and $%
q_{1}=\cdots =q_{m}=\frac{2m}{m+1}$ we recover the classical
Bohnenblust--Hille inequality. The optimal constants $C_{m,p,\mathbf{q}}^{%
\mathbb{K}}$ are known in very few cases, namely

(i) $p=\infty $ and $(q_{1},...,q_{m})=\left( 1,2,...,2\right) $. In this
case (see \cite[Theorem 2.1]{ooo}) $C_{m,\infty ,\mathbf{q}}^{\mathbb{R}%
}=\left( \sqrt{2}\right) ^{m-1}$ for all $m\geq 2;$

(ii) \ $\left( m,p\right) =\left( 2,\infty \right) $ and no restriction on $%
q_{1},q_{2}.$ In this case (see \cite[Theorem 6.3]{alb}) $C_{2,\infty ,%
\mathbf{q}}^{\mathbb{R}}=\sqrt{2}.$

\bigskip

%\textcolor[rgb]{1.00,0.00,0.00}{Este par\'{a}grafo nao eh verdade, deve ser tirado:
%The same argument used to prove (i) also proves that the optimal constants
%associated to $(q_{1},...,q_{m})=\left( 2,...,2,1\right) $ are $\left( \sqrt{2}\right) ^{m-1}$ and thus, all intermediate cases (with $1$ in any position
%among the $2$) the optimal constant is $\left( \sqrt{2}\right) ^{m-1}.$}

In these two cases these optimal constants are obtained by using special
multilinear forms to find lower bounds that match exactly with the known
upper bounds of the constants. This approach seems to be not effective in
other cases, but we do not know if the reason is a fault of the method or a
weakness of our estimates of upper bounds (i.e, maybe the known upper bounds
are not good enough). Using this technique it was recently shown in \cite[%
Theorem 2.3]{ooo} that for a constant $\alpha \in \lbrack 1,2]$ and a
multiple exponent $\mathbf{q}=(\alpha ,\frac{2\alpha m-2\alpha }{\alpha
m-2+\alpha },...,\frac{2\alpha m-2\alpha }{\alpha m-2+\alpha }),$ we have
\begin{equation}
C_{m,\infty ,\mathbf{q}}^{\mathbb{R}}\geq 2^{\frac{2m-\alpha m-4+3\alpha }{%
2\alpha }}.  \label{jfapel}
\end{equation}%
%
%
%
%
%
%
%
%
%
%
%
%
%
%
%
%
%
%
%
%
%
%
%
%
%
%
%
%
%
%
%
%
%
%
%
%
%
%
%
%
%
%
%
%
%
%
%
%
%
%
%
%
%
%
%
%
%
%
%
%
%\textcolor[rgb]{1.00,0.00,0.00}{Nessa anterior estimativa, o professor nao explicita a conta no seu paper, mas pode se obter uma estimativa otima para a constante associada a $m=3$ e $\mathbf{q}=(\frac{4}{3},\frac{8}{5},\frac{8}{5})$ que eh $C_{3,\infty ,\mathbf{q}}^{\mathbb{R}}=2^{3/4}$
%}\newline

By using the Minkowski inequality, we obtain that for $\mathbf{q}=(\frac{%
2\alpha m-2\alpha }{\alpha m-2+\alpha },...,\frac{2\alpha m-2\alpha }{\alpha
m-2+\alpha },\alpha )$, with $\alpha >\frac{2m}{m+1}$ the estimate (\ref%
{jfapel}) gives us

\begin{equation*}
C_{m,\infty ,\mathbf{q}}^{\mathbb{R}}\geq 2^{\frac{2m-\alpha m-4+3\alpha }{%
2\alpha}}.
\end{equation*}

\medskip

In this section, we show that for a constant $\alpha \in \lbrack 1,2]$ and a
$\mathbf{q}=(\frac{2\alpha m-2\alpha }{\alpha m-2+\alpha },...,\frac{2\alpha
m-2\alpha }{\alpha m-2+\alpha },\alpha )$ we have
\begin{equation}
C_{m,\infty ,\mathbf{q}}^{\mathbb{R}}\geq 2^{\frac{3\alpha m-2m-5\alpha +4}{%
2\alpha \left( m-1\right) }}.  \label{thispel}
\end{equation}

For $\alpha >\frac{2m}{m+1}$ the new estimate is strictly bigger (and thus
better) than the previous.

When $m=3$ and $\alpha =2$ we obtain $C_{3,\infty ,\mathbf{q}}^{\mathbb{R}%
}\geq 2^{3/4}$ and since we already know (see \cite[Lemma 2.1]{adv22}) that $%
C_{3,\infty ,\mathbf{q}}^{\mathbb{R}}\leq 2^{3/4}$, we conclude that the
optimal constant is $2^{3/4}.$

The result proved here is:

\begin{proposition}
Let $\alpha \in \lbrack 1,2]$ be a constant and $\mathbf{q}=(\beta
_{m},...,\beta _{m},\alpha )$ be a multiple exponent of the generalized
Bohnenblust--Hille inequality for real scalars. Then
\begin{equation*}
C_{m,\infty ,\mathbf{q}}^{\mathbb{R}}\geq 2^{\frac{3\alpha m-2m-5\alpha +4}{%
2\alpha \left( m-1\right) }}.
\end{equation*}
\end{proposition}

\begin{proof}
The $m$-linear operators that we will use are defined inductively as in (\ref%
{t2}), (\ref{t3}) and (\ref{t4}) . Since
\begin{align*}
{\left( \frac{\left( 2^{m-1}\right) ^{2}}{2}2^{\frac{1}{\alpha }\beta
_{m}}\right) ^{\frac{1}{\beta _{m}}}}& {=}\left( 2^{2m-3}2^{\frac{1}{\alpha }%
\beta _{m}}\right) ^{\frac{1}{\beta _{m}}}=\left( \sum\limits_{i_{1},\ldots
,i_{m-1}=1}^{2^{m-1}}\left( \sum\limits_{i_{m}=1}^{2}\left\vert
T_{m}(e_{i_{^{1}}},...,e_{im})\right\vert ^{\alpha }\right) ^{\frac{1}{%
\alpha }\beta _{m}}\right) ^{\frac{1}{\beta _{m}}} \\
& \leq C_{m,\infty ,\mathbf{q}}^{\mathbb{R}}\left\Vert T_{m}\right\Vert
\end{align*}%
and $\beta _{m}=\frac{2\alpha m-2\alpha }{\alpha m-2+\alpha }$ we conclude
that%
\begin{equation*}
C_{m,\infty ,\mathbf{q}}^{\mathbb{R}}\geq \frac{\left( 2^{2m-3}\left(
2\right) ^{\frac{1}{\alpha }\beta _{m}}\right) ^{\frac{1}{\beta _{m}}}}{%
2^{m-1}}=2^{\frac{3\alpha m-2m-5\alpha +4}{2\alpha \left( m-1\right) }}.
\end{equation*}
\end{proof}

\bigskip

\begin{theorem}
\label{optimal} The optimal constant of the generalized Bohnenblust--Hille
inequality for $m=3$ and $\mathbf{q}=(4/3,4/3,2)$ or $\mathbf{q}%
=(4/3,8/5,8/5)$ or $\mathbf{q}=(4/3,2,4/3)$ is $C_{3,\infty ,\mathbf{q}}^{%
\mathbb{R}}=2^{3/4}$.
\end{theorem}

\begin{proof}
From \cite[Lemma 2.1]{adv22} we obtain, for $m$ and $\mathbf{q}$ satisfying
the hypotheses of the theorem, the estimate%
\begin{equation*}
C_{3,\infty ,\mathbf{q}}^{\mathbb{R}}\leq 2^{3/4}.
\end{equation*}%
Using (\ref{thispel}) we prove that for $\mathbf{q}=(4/3,4/3,2)$ we have%
\begin{equation*}
C_{3,\infty ,\mathbf{q}}^{\mathbb{R}}\geq 2^{3/4}
\end{equation*}%
\medskip and, finally, using (\ref{jfapel}) we show that \ for $\mathbf{q}%
=(4/3,8/5,8/5)$ we have
\begin{equation*}
C_{3,\infty ,\mathbf{q}}^{\mathbb{R}}\geq 2^{3/4}.
\end{equation*}%
On the other hand, using the operator $T_{3}$ (see (\ref{t3})) we have%
\begin{equation*}
2^{\frac{11}{4}}=\left( \sum\limits_{i_{1}=1}^{4}\left(
\sum\limits_{i_{2}=1}^{4}\left( \sum\limits_{i_{3}=1}^{2}\left\vert
T_{3}(e_{i_{^{1}}},e_{i_{2}},e_{i_{3}})\right\vert ^{4/3}\right) ^{\frac{3}{2%
}}\right) ^{\frac{2}{3}}\right) ^{\frac{3}{4}}\text{,}
\end{equation*}%
and thus for $\mathbf{q}=(4/3,2,4/3)$ we get
\begin{equation*}
C_{3,\infty ,\mathbf{q}}^{\mathbb{R}}\geq \frac{2^{\frac{11}{4}}}{2^{2}}=2^{%
\frac{3}{4}}\text{.}
\end{equation*}
\end{proof}

\bigskip

\section{The polynomial Hardy--Littlewood inequality\label{ii1}}

Let $E$ be a real or complex Banach space and $m$ be a positive integer and
let $\mathbb{K}$ be the real or complex scalar field. A map $P:E\rightarrow
\mathbb{K}$ is a homogeneous polynomial on $E$ of degree $m$ if there exists
a symmetric $m$-linear form $L$ on $E^{m}$ such that $P(x)=L(x,\ldots ,x)$
for all $x\in E$. We denote by ${\mathcal{P}}(^{m}E)$ the space of
continuous $m$-homogeneous polynomials on $E$ endowed with the usual norm
\begin{equation*}
\Vert P\Vert :=\sup \{|P(x)|:\left\Vert x\right\Vert =1\}.
\end{equation*}%
Observe that an $m$-homogeneous polynomial in ${\mathbb{K}}^{n}$ can be
written as
\begin{equation*}
P(x)={\sum\limits_{\left\vert \alpha \right\vert =m}}a_{\alpha }x^{\alpha },
\end{equation*}%
where $x=(x_{1},\ldots ,x_{n})\in {\mathbb{K}}^{n}$, $\alpha =(\alpha
_{1},\ldots ,\alpha _{n})\in ({\mathbb{N}}\cup \{0\})^{n}$, $|\alpha
|=\alpha _{1}+\cdots +\alpha _{n}$ and $x^{\alpha }=x_{1}^{\alpha
_{1}}\cdots x_{n}^{\alpha _{n}}$. We denote%
\begin{equation*}
\left\vert P\right\vert _{p}:=\left( {\sum\limits_{\left\vert \alpha
\right\vert =m}}\left\vert a_{\alpha }\right\vert ^{p}\right) ^{1/p}
\end{equation*}%
and%
\begin{equation*}
\left\vert P\right\vert _{\infty }:=\max \left\vert a_{\alpha }\right\vert .
\end{equation*}

\medskip

The polynomial Hardy--Littlewood inequality is:

\textbf{Theorem (Polynomial Hardy--Littlewood inequality).} For $m<p\leq
\infty $ there is a constant $D_{\mathbb{K},m,p}\geq 1$ such that

\begin{equation}
\begin{tabular}{clll}
$\left( {\sum\limits_{\left\vert \alpha \right\vert =m}}\left\vert a_{\alpha
}\right\vert ^{\frac{p}{p-m}}\right) ^{\frac{p-m}{p}}$ & $\leq $ & $D_{%
\mathbb{K},m,p}\left\Vert P\right\Vert $, & $\text{if }m<p\leq 2m,$ \\
&  &  &  \\
$\left( {\sum\limits_{\left\vert \alpha \right\vert =m}}\left\vert a_{\alpha
}\right\vert ^{\frac{2mp}{mp+p-2m}}\right) ^{\frac{mp+p-2m}{2mp}}$ & $\leq $
& $D_{\mathbb{K},m,p}\left\Vert P\right\Vert \text{,}$ & $\text{if }p\geq 2m$%
\end{tabular}
\label{6544}
\end{equation}%
for all positive integers $n$ and all $m$-homogeneous polynomials $P:\ell
_{p}^{n}\rightarrow \mathbb{K}$ given by
\begin{equation*}
P(x)={\sum\limits_{\left\vert \alpha \right\vert =m}}a_{\alpha }x^{\alpha }.
\end{equation*}

This is a consequence of the multilinear Hardy--Littlewood inequality,
previously described, and the following inequality also known as
Hardy--Littlewood inequality \cite{dimant}:

\bigskip

\textbf{Theorem (Hardy--Littlewood/Dimant--Sevilla-Peris).} For $m<p\leq 2m$%
, there is a constant $C_{\mathbb{K},m,p}\geq 1$ such that%
\begin{equation*}
\left( \sum_{i_{1},...,i_{m}=1}^{n}\left\vert T(e_{i_{1}},\ldots
,e_{i_{m}})\right\vert ^{\frac{p}{p-m}}\right) ^{\frac{p-m}{p}}\leq C_{%
\mathbb{K},m,p}\left\Vert T\right\Vert
\end{equation*}%
for all positive integers $n$ and all $m$-linear forms $T:\ell
_{p}^{n}\times \cdots \times \ell _{p}^{n}\rightarrow \mathbb{K}$.

\medskip

Above, the exponent $\frac{p}{p-m}$ is optimal and therefore in (\ref{6544})
both exponents $\frac{p}{p-m}$ and $\frac{2mp}{mp+p-2m}$ are optimal. The
case $p=\infty $ in the appropriate inequality of (\ref{6544}), is the
classical polynomial Bohnenblust--Hille inequality (see \cite{bh}).

\bigskip

From now on $D_{\mathbb{K},m,p}$ denotes the optimal constants satisfying (%
\ref{6544}). As in the multilinear case, the precise behaviour of the growth
of the constants $D_{\mathbb{K},m,p}$ is still unknown (partial results can
be found in \cite{campos, nunez1}). For instance, in \cite[Theorem 3.1]%
{campos} it is proved that for $p\geq 2m$ we have

\bigskip
\begin{equation*}
D_{\mathbb{R},m,p}\geq \left( \sqrt[16]{2}\right) ^{m}.
\end{equation*}

When $p=\infty $ we know that (see \cite{bohr, campos1})
\begin{eqnarray*}
\lim \sup_{m}D_{\mathbb{R},m,\infty }^{1/m} &=&2; \\
\lim \sup_{m}D_{\mathbb{C},m,\infty }^{1/m} &=&1.
\end{eqnarray*}

It will be convenient to define $H_{1}=\left\{ \left( p,m\right) \in \mathbb{%
R}\times \mathbb{N}:m<p<2m\right\} $ and $H_{2}=\left\{ \left( p,m\right)
\in \mathbb{R}\times \mathbb{N}:p\geq 2m\right\} $ with any total order. The
main results of this section are the following:

\begin{lemma}
\label{the:Q_2k} Let $j=1,2.$ Then
\begin{equation*}
\lim \sup_{H_{j}}D_{\mathbb{R},m,p}^{1/m}\geq 2.
\end{equation*}
\end{lemma}

\begin{proof}
Consider the sequence of norm-one $j$-homogeneous polynomials $Q_{j}:\ell
_{p}\rightarrow \mathbb{R}$ defined recursively by
\begin{align*}
Q_{2}(x_{1},x_{2})& =x_{1}^{2}-x_{2}^{2}, \\
Q_{2^{m}}(x_{1},\ldots ,x_{2^{m}})& =Q_{2^{m-1}}(x_{1},\ldots
,x_{2^{m-1}})^{2}-Q_{2^{m-1}}(x_{2^{m-1}+1},\ldots ,x_{2^{m}})^{2}.
\end{align*}%
From the proof of \cite[Theorem 3.1]{campos1}, we known that
\begin{equation}
|Q_{2^{m}}^{n}|_{\infty }\geq \left( \frac{2^{n}}{n+1}\right) ^{2^{m}-1}
\label{eq_m}
\end{equation}%
for every natural number $n,m$.

\noindent Next, since for every homogeneous polynomial $P$ we obviously have
\begin{equation*}
|P|_{p}\geq |P|_{\infty },
\end{equation*}%
from \eqref{eq_m} we conclude that
\begin{equation*}
D_{{\mathbb{R}},n2^{m},p}\geq \left( \frac{2^{n}}{n+1}\right) ^{2^{m}-1}.
\end{equation*}%
Note that

\begin{equation*}
D_{\mathbb{R},n2^{m},p}^{1/n2^{m}}\geq \left( \left( \frac{2^{n}}{n+1}%
\right) ^{2^{m}-1}\right) ^{\frac{1}{n2^{m}}}=\left( \frac{2^{n}}{n+1}%
\right) ^{\frac{2^{m}-1}{n2^{m}}}
\end{equation*}%
and making $m\rightarrow \infty $ we have%
\begin{equation*}
\left( \frac{2^{n}}{n+1}\right) ^{\frac{2^{m}-1}{n2^{m}}}\rightarrow \frac{2%
}{\left( n+1\right) ^{1/n}}
\end{equation*}%
and now making $n\rightarrow \infty $ we have%
\begin{equation*}
\frac{2}{\left( n+1\right) ^{1/n}}\rightarrow 2.
\end{equation*}
\end{proof}

\bigskip From now on we write%
\begin{eqnarray*}
\rho \left( p,m\right) &=&\frac{p}{p-m}\text{ if }m<p\leq 2m, \\
\rho \left( p,m\right) &=&\frac{2mp}{mp+p-2m}\text{ if }p\geq 2m.
\end{eqnarray*}

\medskip Now we prove the theorem:

\bigskip

\begin{theorem}
Let $j=1,2.$ At least one of the following two sentences hold true:
\end{theorem}

(a) $\lim \sup_{H_{j}}D_{\mathbb{R},m,p}^{1/m}=2.$

(b) $\lim \sup_{H_{j}}D_{\mathbb{C},m,p}^{1/m}>1.$

\begin{proof}
Suppose that (a) is not true for some $j$. So (using the previous result) we
would have\linebreak $\lim \sup_{H_{j}}D_{\mathbb{R},m,p}^{1/m}>\left(
2+\varepsilon \right) >2.$ Therefore, for each $k\in \mathbb{N}$ there is $%
n_{k}\in \mathbb{N}$, $\left( p_{_{k}},m_{k}\right) $ $\in $ $H_{j}$ and a $%
m_{k}$-homogeneous polynomial $P_{m_{k}}:\ell _{p_{k}}^{n_{k}}\rightarrow
\mathbb{R}$ such that%
\begin{equation*}
\left( {\sum\limits_{\left\vert \alpha \right\vert =m_{k}}}\left\vert
a_{\alpha }\right\vert ^{\rho \left( p_{_{k}},m_{k}\right) }\right) ^{\frac{1%
}{\rho \left( p_{_{k}},m_{k}\right) }}\leq D_{\mathbb{R},m_{k},p_{k}}\left%
\Vert P_{m_{k}}\right\Vert ,
\end{equation*}%
with%
\begin{equation*}
D_{\mathbb{R},m_{k},p_{k}}>\left( 2+\varepsilon \right) ^{m_{k}}.
\end{equation*}%
Considering the complexification of $P_{m_{k}}$ we know that%
\begin{equation*}
\left\Vert \left( P_{m_{k}}\right) _{\mathbb{C}}\right\Vert \leq
2^{m_{k}-1}\left\Vert P_{m_{k}}\right\Vert
\end{equation*}%
and now looking for the complex polynomials $\left( P_{m_{k}}\right) _{%
\mathbb{C}}$ we would have%
\begin{eqnarray*}
\left( {\sum\limits_{\left\vert \alpha \right\vert =m_{k}}}\left\vert
a_{\alpha }\right\vert ^{\rho \left( p_{k},m_{k}\right) }\right) ^{\frac{1}{%
\rho \left( p_{k},m_{k}\right) }} &\leq &D_{\mathbb{C},m_{k},p_{k}}\left%
\Vert \left( P_{m_{k}}\right) _{\mathbb{C}}\right\Vert \\
&\leq &D_{\mathbb{C},m_{k},p_{k}}2^{m_{k}-1}\left\Vert P_{m_{k}}\right\Vert
\end{eqnarray*}%
and thus%
\begin{equation*}
D_{\mathbb{R},m_{k},p_{k}}\leq D_{\mathbb{C},m_{k},p_{k}}2^{m_{k}-1},
\end{equation*}%
i.e.,\bigskip\
\begin{equation*}
D_{\mathbb{R},m_{k},p_{k}}^{1/m_{k}}\leq D_{\mathbb{C}%
,m_{k},p_{k}}^{1/m_{k}}2^{\frac{m_{k}-1}{m_{k}}}\leq 2D_{\mathbb{C}%
,m_{k},p_{k}}^{1/m_{k}}.
\end{equation*}%
Now, since%
\begin{equation*}
D_{\mathbb{R},m_{k},p_{k}}^{1/m_{k}}>2+\varepsilon
\end{equation*}%
we conclude that%
\begin{equation*}
D_{\mathbb{C},m_{k},p_{k}}^{1/m_{k}}>1+\frac{\varepsilon }{2}>1
\end{equation*}%
for all $k$, and thus
\begin{equation*}
\lim \sup_{H_{j}}D_{\mathbb{C},m,p}^{1/m}>1.
\end{equation*}%
Reciprocally, if (b) is not true for some $j$, then $\lim \sup_{H_{j}}D_{%
\mathbb{C},m,p}^{1/m}=1$ and thus $\lim \sup_{H_{j}}D_{\mathbb{R}%
,m,p}^{1/m}\leq 2$ and from the previous lemma we conclude that
\begin{equation*}
\lim \sup_{H_{j}}D_{\mathbb{R},m,p}^{1/m}=2.
\end{equation*}
\end{proof}

\begin{verbatim}

\end{verbatim}

\textbf{Acknowledgement.} This preprint is the result of the union and
improvement of two arXiv preprints of some of its authors: arXiv 1504.04207
by D.\ Pellegrino and arXiv 1503.00618. The preprint arXiv 1504.04207, as
now is incorporated to the present paper, does not exist as an independent
paper anymore. It is just part of the present new paper.

%%%%%%%%%%%%%%%%%%%

\end{document}